\documentclass[a4paper]{scrartcl}
\usepackage[latin1]{inputenc}
\usepackage{amsmath}
\usepackage{bbm}
\usepackage{amssymb}
\usepackage{amsthm}
\newtheorem{corollary}{Corollary}[section]
\newtheorem{theorem}{Theorem}[section]
\newtheorem*{theorem*}{Theorem}

\newtheorem{conjecture}{Conjecture}[section]
\newtheorem{remark}{Remark}[section]

\numberwithin{equation}{section}

\begin{document}
\title{Complete Monotonicity and Zeros of Sums of Squared Baskakov Functions}
\author{Ulrich Abel
  \thanks{Fachbereich MND Technische Hochschule Mittelhessen, Friedberg, Germany. E-mail: Ulrich.Abel@mnd.thm.de}\\
	Wolfgang Gawronski
	\thanks{Department of Mathematics, University of Trier, Germany. E-mail: gawron@uni-trier.de}\\
	Thorsten Neuschel
	\thanks{Department of Mathematics, KU Leuven, Belgium. E-mail: Thorsten.Neuschel@wis.kuleuven.be}}

 \date{\today}

\maketitle
\begin{abstract} We prove complete monotonicity of sums of squares of generalized Baskakov basis functions by deriving the corresponding results for hypergeometric functions. Moreover, in the central Baskakov case we study the distribution of the complex zeros for large values of a parameter. We finally discuss the extension of some results for sums of higher powers.
\end{abstract}


\section{Introduction}
Recently, there has been an increasing interest in monotonicity and convexity properties of sums of squared basis functions arising in approximation theory (see, e.g., \cite{Gonska}, \cite{Nikolov}, \cite{Gavrea}, \cite{Rasa}). Among the classical approximation operators usually studied in this context we find operators constructed using the basis of Bernstein polynomials, the Mirakjan-Favard-Sz\'{a}sz basis, the Meyer-K\"{o}nig and Zeller basis, the Lagrange basis, the Bleimann-Butzer-Hahn basis and the Baskakov basis. As there is a huge amount of literature on this subject we only refer to \cite{Altomare} for a general introduction. Apart from being of general interest, one motivation for the study of analytic properties of sums of squares of such basis functions is that this proves crucial in a novel approach of Chebyshev-Gr\"{u}ss-type inequalities via discrete oscillations introduced by Gonska, Ra\c{s}a and Rusu in \cite{Gonska} (see also \cite{Gonska0}). These types of inequalities provide estimates for generalized Chebyshev functionals of the form
\[\left\vert L(fg)-L(f)L(g)\right\vert \leq E\left(L,f,g\right),\]
where \(L\) is a linear functional and \(E\) is an expression in terms of properties of \(L\) and some kind of oscillations of \(f\) and \(g\). More precisely, for an arbitrary set \(X\) let 
\(L: B(X) \rightarrow \mathbb{R}\), \(Lf=\sum_{k=0}^{\infty} a_k f(x_k),\) be a positive linear functional defined on the real-valued bounded functions \(B(X)\), where \((x_k)\) is a sequence of distinct points of \(X\) and \(a_k \geq 0\) such that \(\sum_{k=0}^{\infty} a_k =1\). Then we have \cite[Thm. 9]{Gonska}
\begin{equation}\label{CGI}\left\vert L(fg)-L(f)L(g)\right\vert \leq \frac{1}{2} \left(1-\sum_{k=0}^{\infty} a_k^2\right) \mathrm{osc}(f) \mathrm{osc}(g),
\end{equation}
where we define
\[\mathrm{osc}(f)=\sup\left\{\vert f(x_k)-f(x_l) \vert : 0\leq k< l < \infty\right\}\]
and similarly for \(\mathrm{osc}(g)\). Hence, the study of classical approximation operators in the light of the inequality \eqref{CGI} requires a good control of the sums of squares \(\sum_{k=0}^{\infty} a_k^2\). For instance, the investigation of the Bernstein operator lead Gonska, Ra\c{s}a and Rusu to state three interesting conjectures on sums of squared Bernstein functions which recently have been solved and extended using different approaches (see \cite{Gonska}, \cite{Nikolov}, \cite{Gavrea}).

One purpose of this paper is the study of corresponding results for generalized Baskakov bases in the following sense. For \(n\in \mathbb{N}\), \(k \in \mathbb{N}\cup\{0\}\) and a real number \(c>0\) we define the basis functions of Baskakov type (initially introduced in \cite{Baskakov}, see also \cite{Abel}, \cite{Berdy})
\[p_{n,k}^{[c]}(x)=\binom{-n/c}{k}\left(-cx\right)^k \left(1+cx\right)^{-n/c -k},\quad x\geq0.\]
The generalized Baskakov operators defined by
\begin{equation}\label{Op} L_n^{[c]}f (x) = \sum_{k=0}^{\infty}p_{n,k}^{[c]}(x)f\left(\frac{k}{n}\right),\quad x\geq 0,
\end{equation}
form a family of positive operators approximating continuous functions satisfying certain growth conditions. Hence, in this paper we are concerned mainly with the functions
\begin{equation}\label{f1}\psi_{n,c}(x)=\sum_{k=0}^{\infty}\left(p_{n,k}^{[c]}(x)\right)^2,\quad x\geq 0.
\end{equation}
The paper is structured as follows. In Section 2 our first aim is to show the following theorem (see Theorem \ref{C1}).
\begin{theorem*} The function \(\psi_{n,c}\) defined in \eqref{f1} is completely monotonic in the sense of Bernstein's theorem.
\end{theorem*}
This is derived by proving a slightly more general result on a class of generalized hypergeometric functions in Theorem \ref{T1}. A suitable representation as a generalized multivariate Laplace transform turns out to be a crucial tool for the proof. We provide three applications of this. The first one gives a monotonicity result on the complete elliptic integral of the first kind. Next, we deal with the a recent conjecture of I. Ra\c{s}a stating the following (see Conjecture 7.1 in \cite{Rasa}):
\begin{conjecture} The function \(\psi_{n,c}\) is logarithmically convex.
\end{conjecture}
We point out that our result on the complete monotonicity of \(\psi_{n,c}\) verifies the statement of this conjecture in the case \(c>0\). As a third application we obtain a new Chebyshev-Gr\"{u}ss-type inequality for the generalized Baskakov operator defined in \eqref{Op} again for arbitrary \(c>0\). Next, in the central Baskakov case \(c=1\), we study the behavior of all complex zeros of the rational function \(\psi_{n,1}\) for large values of \(n\). This involves some logarithmic potential theory (as introduced in \cite{Saff}) and it is based on a specific representation as a finite hypergeometric sum derived in Theorem \ref{T2}. The main result we will state here is the following (see Theorem \ref{C3}):
\begin{theorem*} All complex zeros of \(\psi_{n,1}\) follow the equilibrium measure on \(D_{\frac{1}{2}}\left(-\frac{1}{2}\right)\), as \(n\rightarrow\infty\). In particular, for large values of \(n\), all the zeros approach the boundary of \(D_{\frac{1}{2}}\left(-\frac{1}{2}\right)\) and they are uniformly distributed in the limit.
\end{theorem*}
Here, \(D_{\frac{1}{2}}\left(-\frac{1}{2}\right)\) denotes the disk around \(-1/2\) with radius \(1/2\). A small third section will be devoted to sums of higher powers of the form
\[\psi_{n,c}^{[r]}(x)=\sum_{k=0}^{\infty}\left(p_{n,k}^{[c]}(x)\right)^r,\quad x\geq 0,\]
where \(r>1\) is an arbitrary integer. In Theorem \ref{L1} we will generalize the integral representation as a multivariate Laplace type transform (see Remark \ref{R1}) to arbitrary powers \(r\). Although this representation seems to reveal the complete monotonicity of \(\psi_{n,c}^{[r]}\) (for even positive integers \(r\)) without additional effort only in the case \(r=2\), we use it to derive the decay at infinity of the derivatives of all orders (see Remark \ref{R2}). Finally, we state a conjecture about the complete monotonicity of \(\psi_{n,c}^{[r]}\) (and of a related generalized hypergeometric function) in the case \(r\) is a positive even integer.

\section{Sums of squared Baskakov functions}
The first aim in this section is to show that \(\psi_{n,c}\) defined in \eqref{f1} is a completely monotonic function in the sense of Bernstein's theorem. To this end, we first prove the following theorem.
\begin{theorem}\label{T1} For \(\alpha >0\) we define 
\[f_{\alpha}(x)=(1+x)^{-2\alpha}\sum_{k=0}^{\infty}\binom{-\alpha}{k}^2\left(\frac{x}{1+x}\right)^{2k}, \quad x\geq 0.\]
Then \(f_{\alpha}\) is a completely monotonic function.
\end{theorem}
\begin{proof} The proof can be based on a representation of \(f_{\alpha}\) by means of a triple integral of Laplace type. Using the notion of generalized hypergeometric functions we can write
\[f_{\alpha}(x)=(1+x)^{-2\alpha}~  _{2} F_{1} \left(\begin{matrix}
  \alpha,&\alpha \\ 1 \end{matrix}\,\bigg\vert\,  \left(\frac{x}{1+x}\right)^{2} \right)=(1+x)^{-2\alpha} \sum_{k=0}^{\infty}\frac{(\alpha)^2_k}{k!^2}\left(\frac{x}{1+x}\right)^{2k}.
\]
Substituting the two Pochhammer symbols by integrals of the form
\[(\alpha)_k =\frac{1}{\Gamma (\alpha)}\int\limits_{0}^{\infty}e^{-t} t^{\alpha+k-1} dt,\]
we obtain after some algebra
\begin{align*}f_{\alpha}(x)&=\frac{(1+x)^{-2\alpha}}{\Gamma(\alpha)^2}\int\limits_{(0,\infty)^2} e^{-(s+t)} (st)^{\alpha-1} \sum_{k=0}^{\infty}\frac{1}{k!^2}\left(st \left(\frac{x}{1+x}\right)^2\right)^{k} d(s,t)\\
&=\frac{(1+x)^{-2\alpha}}{\Gamma(\alpha)^2}\int\limits_{(0,\infty)^2} e^{-(s+t)} (st)^{\alpha-1} I_0 \left(2\sqrt{st} \frac{x}{1+x}\right)d(s,t),
\end{align*}
where \(I_0\) denotes the modified Bessel function of order \(0\). Now, using the well-known integral representation
\begin{equation}\label{Int1}I_0(y)=\frac{1}{\pi}\int\limits_{0}^{\pi}e^{-y \cos(\theta)} d\theta
\end{equation}
we get
\begin{align}\nonumber f_{\alpha}(x)&=\frac{(1+x)^{-2\alpha}}{\pi \Gamma(\alpha)^2}\int\limits_{(0,\infty)^2} \int\limits_{0}^{\pi} e^{-(s+t)} (st)^{\alpha-1} \exp\left(-2\sqrt{st}\frac{x}{1+x} \cos(\theta)\right) d(s,t) d\theta\\\label{LT1}
&=\frac{1}{\pi \Gamma(\alpha)^2}\int\limits_{(0,\infty)^2} \int\limits_{0}^{\pi} e^{-(s+t+2\sqrt{st}\cos\theta)x}  e^{-(s+t)}(st)^{\alpha-1}  d(s,t) d\theta,
\end{align}
where the last equality results from a simple change of the variables \(s\) and \(t\). Observing that we have \(s+t+2\sqrt{st}\cos\theta \geq 0\), by differentiation of \eqref{LT1} with respect to \(x\) we obtain for every \(m\in\mathbb{N}\cup\{0\}\)
\[(-1)^m \left(\frac{d}{dx}\right)^m f_{\alpha}(x) > 0, \quad x\geq 0.\]
\end{proof}
\begin{remark}\label{R1}The identity \eqref{LT1} in the proof of Theorem \ref{T1} can be considered as a representation of \(f_{\alpha}\) as a generalized multivariate Laplace transform. Moreover, by Bernstein's theorem, it follows from the statement of Theorem \ref{T1} that there exists a representation of \(f_{\alpha}\) as a univariate Laplace transform in the usual sense.
\end{remark}
\begin{remark} An application of Theorem \ref{T1} in the case \(\alpha=1/2\) shows the complete monotonicity of the function
\[\frac{1}{1+x}K\left(\frac{x}{1+x}\right), \quad x\geq 0,\]
where \(K\) denotes the complete elliptic integral of the first kind (for the definition see, e.g., \cite[p.\ 487]{NIST}). This follows from identity
\[K(x)= \frac{\pi}{2}\  _{2} F_{1} \left(\begin{matrix}
  1/2,&1/2 \\ 1 \end{matrix}\,\bigg\vert\,  x^2 \right),\quad x\in[0,1).\]
An application of Bernstein's theorem then provides an integral representation of the form
\[K(x) = \frac{1}{1-x}\int\limits_{0}^{\infty} e^{-\frac{x}{1-x} t} d\mu(t),\quad x\in[0,1),\]
for a finite positive Borel measure \(\mu\). This representation may prove useful in the study of the function \(K\). There is a serious amount of literature on elliptic integrals, for monotonicity properties see, e.g., \cite{Alzer}.
\end{remark}
In regards to Baskakov functions, from Theorem \ref{T1} we now obtain the following theorem.
\begin{theorem}\label{C1} For \(n\in\mathbb{N}\) and \(c>0\) the function \(\psi_{n,c}\) defined in \eqref{f1} is completely monotonic.
\end{theorem}
\begin{proof} We have
\[\psi_{n,c}(x)=(1+cx)^{-2n/c}\sum_{k=0}^{\infty}\binom{-n/c}{k}^2\left(\frac{cx}{1+cx}\right)^{2k}.\]
Putting \(\alpha=n/c\) in Theorem \ref{T1} we know that the function \(\psi_{n,c}\left(\frac{x}{c}\right)\) is completely monotonic. Hence, the same is true for the function \(\psi_{n,c}\) as well.
\end{proof}
\begin{remark} \label{R3}The limiting case \(c\rightarrow 0\) corresponds to the sum of the squared Mirakjan-Favard-Sz\'{a}sz basis functions. The statement of Theorem \ref{C1} remains true in this case as was shown by Gavrea and Ivan in \cite{Gavrea}. This follows from the representation
\[\psi_{n,0}(x)=\sum_{k=0}^{\infty}\left(\frac{e^{-nx} (nx)^k}{k!}\right)^2=\frac{2}{\pi}\int\limits_{0}^{\pi/2} e^{-4n x \sin^2 t} dt.\]
In the case \(c=-1\) the functions \(\psi_{n,c}\) can be interpreted as sums of squares of the Bernstein basis. For a related result in this case see Theorem 3 together with Remark 1 in \cite{Gavrea}.
\end{remark}
From the general fact that completely monotonic functions are logarithmically convex (see, e.g., Theorem. 1 in \cite{Fink}) we immediately obtain the following corollary.
\begin{corollary} For \(n\in\mathbb{N}\) and \(c>0\) the function \(\psi_{n,c}\) is logarithmically convex.
\end{corollary}
This was conjectured recently (but not necessarily only for \(c>0\)) by I. Ra\c{s}a in Conjecture 7.1 in \cite{Rasa}.

As a further application of Theorem \ref{C1} we obtain an inequality of Chebyshev-Gr\"{u}ss-type for the positive linear operator \(L_n^{[c]}\) defined in \eqref{Op}. We point out that this provides an extension of Theorem 16 in \cite{Gonska} from the special case \(c=1\) to arbitrary \(c>0\).
\begin{corollary}\label{C2} We have for bounded functions \(f, g:[0,\infty) \rightarrow \mathbb{R}\)
\[\left\vert L_n^{[c]}fg (x) -L_n^{[c]}f (x) L_n^{[c]}g (x) \right\vert \leq \frac{1}{2} \mathrm{osc}(f) \mathrm{osc}(g), \quad x\geq 0,\]
where 
\[\mathrm{osc}(f)=\sup\left\{\left\vert f\left(\frac{k}{n}\right)-f\left(\frac{l}{n}\right) \right\vert : 0\leq k< l < \infty\right\}\]
and similarly for \(\mathrm{osc}(g)\).
\end{corollary}
\begin{proof}
For \(n\in\mathbb{N}\) and \(c>0\) we have
\[\inf_{x\in [0,\infty)}\psi_{n,c}(x)=\lim_{x\rightarrow \infty}\psi_{n,c}(x)=0.\]
This easily follows from the representation of \(\psi_{n,c}\) as a Laplace transform, which exists according to Bernstein's theorem in connection with Theorem \ref{C1}. Thus, the statement follows from the inequality \eqref{CGI}.
\end{proof}

Next, we will study the behavior of the (complex) zeros of the function \(\psi_{n,c}\) in the central case \(c=1\) for large values of \(n\). To this end, we use a specific representation by means of hypergeometric functions.

\begin{theorem}\label{T2} We have for all \(n\in\mathbb{N}\)
\begin{equation}\label{eq1}\sum_{k=0}^{\infty}\left(p_{n,k}^{[1]}(x)\right)^2 =\frac{1}{1+2x}\binom{n-3/2}{n-1}~_{2} F_{1} \left(\begin{matrix}
  -n+1,&\frac{1}{2} \\ -n+\frac{3}{2} \end{matrix}\,\bigg\vert\,  (1+2x)^{-2} \right).
\end{equation}
\end{theorem}
\begin{proof} Starting from the identity
\[\sum_{k=0}^{\infty}p_{n,k}^{[1]}(x) e^{ikt} = \left(1+x-xe^{it}\right)^{-n}, \quad t\in\mathbb{R},\]
an application of Parseval's identity yields
\[\sum_{k=0}^{\infty}\left(p_{n,k}^{[1]}(x)\right)^2 = \frac{1}{2\pi}\int\limits_{-\pi}^\pi \left\vert 1+x-xe^{it}\right\vert^{-2n} dt= \frac{1}{2\pi}\int\limits_{-\pi}^\pi \left(1+2x(1+x)(1-\cos t)\right)^{-n} dt.\]
Now, the change of variable \(u=\tan\left(t/2\right)\) gives us
\begin{align*}\sum_{k=0}^{\infty}\left(p_{n,k}^{[1]}(x)\right)^2 =& \frac{1}{\pi}\int\limits_{-\infty}^{\infty} \left(1+2x(1+x)\frac{2u^2}{1+u^2}\right)^{-n} \frac{du}{1+u^2}\\
=&\frac{2}{\pi}\int\limits_{0}^{\infty} \frac{(1+u^2)^{n-1}}{\left(1+(1+2x)^2 u^2\right)^n} du.
\end{align*}
A further simple change of variables and an application of the binomial theorem yields
\[\sum_{k=0}^{\infty}\left(p_{n,k}^{[1]}(x)\right)^2=\sum_{j=0}^{n-1} c_{n,j} \left(1+2x\right)^{-2j-1},\]
where the coefficients are given by
\[c_{n,j}=\frac{2}{\pi}\binom{n-1}{j} \int\limits_{0}^{\infty}\frac{u^{2j}}{\left(1+u^2\right)^n} du.\]
Finally, we will compute the coefficients explicitly by using the Beta function \(B\). From
\[\int\limits_{0}^{\infty}\frac{u^{2j}}{\left(1+u^2\right)^n} du=\frac{1}{2} B\left(j+\frac{1}{2}, n-j-\frac{1}{2}\right)=\frac{\pi (2j)! (2n-2j-2)!}{2^{2n-1} (n-1)! j! (n-1-j)!}\]
we obtain 
\[c_{n,j}= \frac{ (2j)! (2n-2j-2)!}{2^{2n-2}  \left(j! (n-1-j)!\right)^2},\]
which implies
\[\sum_{k=0}^{\infty}\left(p_{n,k}^{[1]}(x)\right)^2=\frac{1}{2^{2n-2}} \binom{2n-2}{n-1} \sum_{j=0}^{n-1} \binom{n-1}{j}^2 \binom{2n-2}{2j}^{-1} \left(1+2x\right)^{-2j-1}.\]
This latter expression can easily be rewritten in terms of hypergeometric functions which yields the representation \eqref{eq1}.
\end{proof}
\begin{remark} The idea to apply Parseval's identity to obtain a suitable hypergeometric representation is used by Gavrea and Ivan in \cite{Gavrea}. In fact, in Theorem 5 they derive a representation for the sums of squares of the classical Meyer-K\"{o}nig and Zeller basis \(b_{n,k}(x)=\binom{n+k}{k}x^k (1-x)^{n+1}\), which is connected to the central Baskakov case by
\[p_{n+1,k}^{[1]}(x)=b_{n,k}\left(\frac{x}{1+x}\right),\quad x\in[0,\infty).\]
\end{remark}
\begin{remark} The identity \eqref{eq1} in Theorem \ref{T2} provides an alternative proof of the complete monotonicity of \(\psi_{n,c}\) in the case \(c=1\) as it expresses \(\psi_{n,1}\) as a sum of odd powers of \((1+2x)^{-1}\) with positive coefficients.
\end{remark}
Let the sequence of hypergeometric polynomials \((P_n)\) be defined by
\[P_n(z)=~ _{2} F_{1} \left(\begin{matrix}
  -n,&\frac{1}{2} \\ -n+\frac{1}{2} \end{matrix}\,\bigg\vert\,  z \right), \quad n\in\mathbb{N}.\]
It is known \cite{Sri} that the zeros of \(P_n\) approach the unit circle as \(n\) tends to infinity. In addition to that, we are going to show that in the limit, the zeros are distributed uniformly on the unit circle. To this end, we use techniques from logarithmic potential theory (cf. \cite{Abel2} and \cite{Neuschel}).

\begin{theorem} \label{Thm1}The zeros of the polynomials \(P_n\) follow the equilibrium measure on the unit disk. More precisely, if \((\mu_n)\) denotes the sequence of normalized zero counting measures associated to \(P_n\), then the sequence \((\mu_n)\)  converges in the weak-star sense to the uniform distribution on the unit circle (normalized arc measure).
\end{theorem}
\begin{proof} We begin by recapitulating the asymptotic behavior of \(P_n\) from the proof of Theorem 1 in \cite{Sri}. For every \(0<\epsilon<1\) we have
\[\lim_{n\rightarrow\infty} P_n(z)=(1-z)^{-\frac{1}{2}}\]
and
\begin{equation}\label{1}\lim_{n\rightarrow\infty} z^n P_n\left(\frac{1}{z}\right)=(1-z)^{-\frac{1}{2}}
\end{equation}
uniformly on \(\vert z\vert \leq \epsilon\). From the fact that the limits are free from any zeros we infer that, if the sequence \((\mu_n)\) has a weak-star limit, then it will be supported on the unit circle. Now, if we start with an arbitrary subsequence of \((\mu_n)\), then by Helly's selection principle \cite{Saff} we can choose a further subsequence, denoted by \((\mu_{n_k})\), which has a weak-star limit \(\mu\). For the logarithmic potentials we know
\[\mathcal{U}^{\mu_n}(z)=-\frac{1}{n}\log \vert P_n(z)\vert = -\log\vert z \vert -\frac{1}{n} \log \vert z^{-n} P_n(z)\vert.\]
Now, for \(\vert z\vert >1\) we obtain using (\ref{1}) that
\[\lim_{n\rightarrow\infty} \frac{1}{n} \log \vert z^{-n} P_n(z)\vert = 0,\]
which gives
\[\mathcal{U}^{\mu}(z) =\lim_{k\rightarrow\infty}\mathcal{U}^{\mu_{n_k}}(z)=\log\frac{1}{\vert z\vert}.\]
Observing that this coincides with the logarithmic potential of the uniform distribution (equilibrium measure) of the unit disk, using Carleson's unicity theorem \cite{Saff} we can conclude that \(\mu\) coincides with the uniform distribution. Moreover, this shows that the whole sequence \((\mu_n)\) converges in the weak-star sense to the uniform distribution, which completes the proof.
\end{proof}

Next, using the representation from Theorem \ref{T2}
\[\sum_{k=0}^{\infty}\left(p_{n,k}^{[1]}(x)\right)^2 =\frac{1}{1+2x}\binom{n-3/2}{n-1} P_{n-1}\left(\frac{1}{(1+2x)^2}\right)\]
we immediately can translate Theorem \ref{Thm1} into a statement for the squared Baskakov functions. By \(D_{\frac{1}{2}}\left(-\frac{1}{2}\right)\) we denote the disk around \(-\frac{1}{2}\) with radius \(\frac{1}{2}\).
\begin{theorem} \label{C3}All complex zeros of the rational functions 
\[\psi_{n,1}(x)=\sum_{k=0}^{\infty}\left(p_{n,k}^{[1]}(x)\right)^2 \]
follow the equilibrium measure on \(D_{\frac{1}{2}}\left(-\frac{1}{2}\right)\), as \(n\rightarrow\infty\). In particular, for large values of \(n\), all the zeros approach the boundary of \(D_{\frac{1}{2}}\left(-\frac{1}{2}\right)\) and they are uniformly distributed in the limit.
\end{theorem}
\begin{remark} The associated Legendre differential equation is defined by the expression
\[(1-x)^2 y'' -2xy+\left(v(v+1)-\frac{u^2}{1-x^2}\right) y =0,\]
where \(u,v\) are (real) parameters. The solutions are called associated Legendre functions of the first kind (or second kind respectively). If we denote the solutions of the first kind by \(y(v,u;x)\), then it turns out that we can connect the squared Baskakov functions to the associated Legendre functions by the identity
\[\sum_{k=0}^{\infty} \left(p_{n,k}^{[1]}(x)\right)^2=\frac{(-1)^{n+1} \sqrt{\pi} (2x+1)^{-n+1/2}}{2 (n-1)! \sqrt{x(x+1)}}y\left(-\frac{1}{2},n-\frac{1}{2}; \frac{2x^2+2x+1}{2x(x+1)}\right).\]
This can be considered as an analogue to the appearence of the Legendre polynomials in the Bernstein case (see \cite{Nikolov}).
\end{remark}

\section{Sums of higher powers}
This section is devoted to more general sums of powers of the form
\[\psi_{n,c}^{[r]}(x)=\psi_{n,c}^{[r]}(x)=\sum_{k=0}^{\infty}\left(p_{n,k}^{[c]}(x)\right)^r,\quad x\geq 0,\]
where \(r>1\) is an arbitrary integer. Numerical simulations suggest that the functions \(\psi_{n,c}^{[r]}\) are completely monotonic if \(r\) is an even integer. Trying to prove this by the approach used in the case \(r=2\) would lead to an analogue of Theorem \ref{T1}, especially of the integral representation \eqref{LT1}.
\begin{theorem}\label{L1}For an integer \(r>1\) and real \(\alpha>0\) we define
\[f_{\alpha}^{[r]}(x)=(1+x)^{-r\alpha} \sum_{k=0}^{\infty}\binom{-\alpha}{k}^r\left(\frac{x}{1+x}\right)^{rk}, \quad x\geq 0.\]
Then the following representation as a generalized multivariate Laplace transform holds
\begin{equation}\label{int2}f_{\alpha}^{[r]}(x)=\frac{1}{\left(2\pi\right)^{r-1} \Gamma(\alpha)^r}\int\limits_{(0,\infty)^r}\int\limits_{(-\pi,\pi)^{r-1}}e^{-x g(t,\varphi)}\exp\left\{-\sum_{j=1}^r t_j\right\} \prod_{j=1}^r t_j^{\alpha-1}dtd\varphi,
\end{equation}
where \(dt=d(t_1,\ldots,t_r)\), \(d\varphi=d(\varphi_1,\ldots,\varphi_{r-1})\) and 
\begin{equation}\label{g}g(t,\varphi)=\sum_{j=1}^r t_j +\prod_{j=1}^r t_j^{1/r}\left\{\sum_{j=1}^{r-1} e^{i\varphi_j} +\exp\left(-i\sum_{j=1}^{r-1} \varphi_j\right)\right\}.
\end{equation}
\end{theorem}
\begin{proof}We have 
\[f_{\alpha}^{[r]}(x)=(1+x)^{-r\alpha}~ _{r} F_{r-1} \left(\begin{matrix}
  \alpha,&\ldots,&\alpha \\ 1,&\ldots,&1 \end{matrix}\,\bigg\vert\,  \left(\frac{-x}{1+x}\right)^{r} \right)=(1+x)^{-r\alpha} \sum_{k=0}^{\infty}\frac{(\alpha)^r_k}{(k!)^r}\left(\frac{-x}{1+x}\right)^{rk}.\]
We begin like in the proof of Theorem \ref{T1} by substituting all the Pochhammer symbols by integrals of the form
\[(\alpha)_k =\frac{1}{\Gamma (\alpha)}\int\limits_{0}^{\infty}e^{-t_j} t_j^{\alpha+k-1} dt_j\]
giving
\begin{align}\nonumber f_{\alpha}^{[r]}(x)&=\frac{(1+x)^{-r\alpha}}{\Gamma(\alpha)^r}\int\limits_{(0,\infty)^r} e^{-(t_1+\ldots+t_r)} (t_1\cdots t_r)^{\alpha-1} \sum_{k=0}^{\infty}\frac{1}{(k!)^r}\left(t_1\cdots t_r \left(\frac{-x}{1+x}\right)^r\right)^{k} dt\\\label{2}
&=\frac{1}{\Gamma(\alpha)^r}\int\limits_{(0,\infty)^r} e^{-(1+x)(t_1+\ldots+t_r)} (t_1\cdots t_r)^{\alpha-1} I^{[r]} \left(t_1\cdots t_r (-x)^r\right)dt,
\end{align}
where the last equality results from a simple change of variables and the function \(I^{[r]}\) is a generalization of the Bessel (and the exponential) function defined by
\[I^{[r]}(x)=\sum_{k=0}^{\infty}\frac{x^k}{(k!)^r}.\]
Next, we use an analogue of \eqref{Int1} for \(I^{[r]}\) which we derive by introducing Hankel's integral representation 
\[\frac{1}{k!}=\frac{1}{2\pi i}\int\limits_\gamma e^{z} z^{-k-1} dz,\]
where the integration is extended over a closed positively oriented contour in the complex plane around the origin \(\gamma\). This way, we can replace the factor \(1/(k!)^r\) by means of an \((r-1)\)-fold complex contour integral and obtain
\[I^{[r]} \left(t_1\cdots t_r (-x)^r\right)=\frac{1}{(2\pi i)^{r-1}}\int\limits_{\gamma\times \ldots \times \gamma}e^{z_1+\ldots +z_{r-1}} \exp\left\{\frac{t_1\cdots t_r (-x)^r}{z_1\cdots z_{r-1}}\right\}\frac{dz}{z_1\cdots z_{r-1}},\]
where \(dz=d(z_1\cdots z_{r-1})\). Now, using the parameterizations 
\[z_j=-(t_1\cdots t_r)^{1/r}xe ^{i\varphi_j},\quad j=1,\ldots,r-1,\]
we obtain for \(I^{[r]} \left(t_1\cdots t_r (-x)^r\right)\) the \((r-1)\)-fold integral representation
\begin{equation}\label{3}\frac{1}{(2\pi)^{r-1}}\int\limits_{(-\pi,\pi)^{r-1}}\exp\left\{-x(t_1\cdots t_r)^{1/r}\left(\sum_{j=1}^{r-1} e^{i\varphi_j} +\exp\left(-i\sum_{j=1}^{r-1} \varphi_j\right)\right)\right\}d\varphi,
\end{equation}
where \(d\varphi=d(\varphi_1,\ldots,\varphi_{r-1})\). Using \eqref{3} in \eqref{2} we immediately obtain \eqref{int2}.	
\end{proof}
\begin{remark}\label{R2}While numerical simulations suggest that the functions \(f_{\alpha}^{[r]}\) are completely monotonic if \(r\) is an even positive integer, the integral representation in Theorem \ref{L1} seems to reveal this property without additional effort only in the case \(r=2\) (reducing to the tripel integral \eqref{LT1}) as shown in Theorem \ref{T1}. However, as an immediate consequence of the inequality of arithmetic and geometric means, we know that the real part of the function \(g\) defined in \eqref{g} satisfies
\[\Re\left\{ g(t,\varphi) \right\}\geq 0\]
with equality attained only on a set of measure zero. Hence, as an application of \eqref{int2} and using Lebesgue's dominated convergence theorem we get the limiting behavior of all the derivatives of \(f_{\alpha}^{[r]}\), i.e., for every \(m\in\mathbb{N}\cup\{0\}\) (and every integer \(r>1\) and real \(\alpha>0\)) we have
\[\lim_{x\rightarrow \infty}\left(\frac{d}{dx}\right)^m f_{\alpha}^{[r]}(x)=0.\]
This remark immediately can be reformulated for the functions \(\psi_{n,c}^{[r]}(x)\). If \(c\) is a positive real number and \(n,r\) are positive integers, then for every \(m\in\mathbb{N}\cup\{0\}\) we have
\[\lim_{x\rightarrow \infty}\left(\frac{d}{dx}\right)^m \psi_{n,c}^{[r]}(x)=0.\]
\end{remark}
We close this section by stating a conjecture.
\begin{conjecture} Let \(r>1\) be an even integer and let \(\alpha\) be a real positive number. Then
\[f_{\alpha}^{[r]}(x)=(1+x)^{-r\alpha} \sum_{k=0}^{\infty}\binom{-\alpha}{k}^r\left(\frac{x}{1+x}\right)^{rk}, \quad x\geq 0,\]
is a completely monotonic function on \([0,\infty)\). In particular, if \(r>1\) is an even integer, \(c\) is a positive real number and \(n\in\mathbb{N}\), then  
\[\psi_{n,c}^{[r]}(x)=\sum_{k=0}^{\infty}\left(p_{n,k}^{[c]}(x)\right)^r,\quad x\geq 0,\]
is a completely monotonic function on \([0,\infty)\).
\end{conjecture}
\section*{Acknowledgements} Thorsten Neuschel gratefully acknowledges support from KU Leuven research grant OT\slash12\slash073 and the Belgian Interuniversity Attraction Pole P07/18.

\end{document}